\documentclass[11pt,amstex,amssymb]{amsart}
\usepackage{amsmath,amsthm,amsfonts,amssymb,amscd}
\usepackage[latin1]{inputenc}
\usepackage[all]{xy}
\usepackage[dvips]{graphicx}
\usepackage{amsmath}
\usepackage{amsthm}
\usepackage{amsfonts}
\usepackage{amssymb}

\usepackage{amsmath,amsthm,amsfonts,amssymb,amscd}
\usepackage[latin1]{inputenc}

\usepackage{graphicx}
\input xy
\xyoption{all}
\usepackage[all]{xy}
\theoremstyle{plain}\newtheorem{Theorem}{Theorem}[section]
\newtheorem{Lemma}{Lemma}[section]
\newtheorem{Proposition}{Proposition}[section]
\newtheorem{Definition}{Definition}[section]
\newtheorem{Corollary}{Corollary}[section]
\newtheorem{Example}{Example}[section]
\newtheorem{Remark}{Remark}[section]

\newtheorem{Claim}{Claim}[section]

\newcommand{\re}{{\mathbb R}}

\newcommand{\co}{{\mathbb C}}

\newcommand{\la}{\lambda}
\newcommand{\La}{\Lambda}

\def\fa{{\mathcal{F}}}
\def\re{{\mathbb{R}}}

\def\po{{\partial}}

\def\La{{\Lambda}}

\def\Om{{\Omega}}

\def\GL{\operatorname{{GL}}}

\def\Aff{\operatorname{{Aff}}}
\def\dim{\operatorname{{dim}}}

\def\Om{{\Omega}}

\def\la{{\lambda}}
\def\La{{\Lambda}}

\def\re{{\mathbb{R}}}

\def\sep{\operatorname{{sep}}}
\def\Ker{\operatorname{{Ker}}}

\def\Ker{\operatorname{{Ker}}}
\def\Ker{\operatorname{{Ker}}}
\def\Ker{\operatorname{{Ker}}}
\def\Ker{\operatorname{{Ker}}}

\def\sing{\operatorname{{sing}}}

\def\Om{{\Omega}}

\def\la{{\lambda}}
\def\La{{\Lambda}}

\def\re{{\mathbb{R}}}

\def\h.o.t.{\operatorname{{h.o.t.}}}

\def\sing{\operatorname{{sing}}}

 \def\_#1{{\lower 0.7ex\hbox{}}_{#1}}

 \def\la{{\lambda}}

 \def\re{{\mathbb{R}}}

\def\Aff{\operatorname{{Aff}}}
\def\aff{\operatorname{{aff}}}

\def\sep{\operatorname{{sep}}}

\def\GL{\operatorname{{GL}}}
\def\DL{\operatorname{{DL}}}

\def\X{{\mathcal  X}}

\title[Transversely affine holomorphic foliations - I]{Transversely affine holomorphic foliations of arbitrary codimension - I}

\author{B. Sc\'ardua}

\address{B. Scardua: Inst. Matem\'atica,  Universidade Federal do Rio de Janeiro.   Caixa Postal 68530,
 Rio de Janeiro-RJ,  21.945-970 BRAZIL}

\date{}
\begin{document}

\maketitle

\begin{abstract}
We study holomorphic foliations with an affine homogeneous transverse structure. We give a  friendly characterization of  the case of transversely affine foliations in terms of matrix valued pairs of differential forms. This leads naturally to the study of the case  of foliations with singularities. A first extension theorem is then proved in the generic singularities framework.
\end{abstract}

%\tableofcontents

\thispagestyle{empty}

\section{Introduction}

The study of the geometry of foliations often is related to the study of their transverse structure. Among the most comprehensible structures are those given by
actions of Lie groups on some homogeneous space. This is the case of the so called {\it transversely homogeneous foliations} as introduced by Blumenthal (\cite{Blumenthal,Godbillon}. One of the first cases of such a class of foliations, is the class of transversely affine foliations. Such foliations have been studied in the smooth  real codimension one case by Bobo Seke in \cite{boboseke}. In \cite{scardua} the author  considers the case of codimension one holomorphic foliations with singularities. A classification is given for such objects on complex projective spaces.

In this paper we consider the case of arbitrary codimension. We focus on the holomorphic case, already aiming the case of foliations with singularities. Nevertheless, most of the material in the first sections  also holds in the (non-singular) smooth case.
In few words, our aim is to introduce the first ingredients in the study of the case of transversely homogeneous holomorphic foliations with singularities.

\subsection{Transversely affine foliations}

Let us clearly state the notions we use. The following definition is found in \cite{Blumenthal} or in \cite{Godbillon} pp. 245. We adapt it to the holomorphic case:

\begin{Definition}[transversely homogeneous foliation]
 {\rm Let $\mathcal F$ be a holomorphic  foliation on a
complex manifold $P$. Let $G$ be a simply-connected Lie group and $H
\subset G$ be a connected closed subgroup of $G$. We say that
$\mathcal F$ is {\it transversely homogeneous\/} in $P$ of model
$G/H$ if $P$ admits an open cover $\bigcup
\limits_{i \in I} U_i = P$ with holomorphic submersions
$y_i\colon U_i \to G/H$ satisfying: (i) $\mathcal F\big|_{U_i}$ is
defined by $y_i$, (ii) In each $U_i \cap U_j \ne \emptyset$ we have $y_i
= g_{ij}\circ y_j$ for some locally constant map $g_{ij}\colon U_i \cap
U_j \to G$. }
\end{Definition}
Notice that the group $G$ acts on the quotient $P=G/H$ by left translations.
In particular, we have:

\begin{Definition}
\label{Definition:transvaffine} {\rm A holomorphic  codimension-$q$ foliation ${\mathcal F}$ on
$M^n$ is  {\it transversely affine\/} if there is a family
$\{Y_i\colon U_i \to \co^q\}\_{i\in I}$ of holomorphic submersions
$Y_i\colon U_i \to \co^q$ defined in open sets $U_i \subset M$,
defining ${\mathcal F}\big|_{U_i}$,  covering  $M
= \bigcup\limits_{i \in I} \,U_i$ and such that for each $U_i \cap U_j \ne\phi$ we have  $Y_i =
A_{ij}Y_j + B_{ij}$ for some locally constant maps $A_{ij}\colon U_i\cap U_j \to \GL_q(\co)$, $B_{ij}\colon U_i\cap U_j \to \co^q$. }
\end{Definition}

\subsection{Integrable systems and foliations}
Recall that  a system of holomorphic 1-forms $\Omega:= \{\Omega_1,...,\Omega_q\}$ in an open set $U\subset M$ is {\it integrable} if for every $j \in \{1,...,q\}$ we have $d\Omega_j \wedge \Omega_1\wedge \ldots \wedge \Omega_q=0$ in $U$. If such a system of forms has maximal rank at each point, then it defines a codimension $q$ holomorphic foliation $\fa(\Omega)$ on $U$. The foliation is given by the integrable distribution of $(n-q)$-planes $\Ker(\Omega):= \bigcap \limits_{j=1} ^q \Ker(\Omega_j)$ where given $p \in M$ we define $\Ker(\Omega_j)(p) :=\{ v \in T_p(M) : \Omega_j(p) \cdot v =0\}$. Two such maximal rank integrable systems $\Omega$ and $\Omega^\prime$ define the same foliation in $U$ if, and only if, we have $\Omega_i = \sum\limits_{j=1} ^q a_{ij} \Omega_j$ for some holomorphic functions $a_{ij}$ in $U$, with the property that the $q\times q$ matrix $A=(a_{ij})_{i,j=1}^q$ is nonsingular at each point of $U$.
Given a system $\{\Omega_1,...,\Omega_q\}$ as above, we define a $q\times 1$ matrix valued 1-form $\Omega$ as having rows given by $\Omega_1,...,\Omega_q$. 
We denote by $\fa(\Omega)$ the foliation defined by this system.

Let us now introduce some notation.
Given a $k \times \ell$ and a $\ell \times s$ matrix valued 1-form
$A=(a_{ij})$ and $B=(b_{jt})$ respectively, we may define the wedge product $A \wedge B$ in the natural way, as the $k \times s$ matrix valued 1-form $A\wedge B$ whose
entry at the position $(i,t)$ is the 2-form $\sum\limits_{j=1} ^\ell a_{ij} \wedge b_{jt}$. In the same way we may define the exterior derivative $dA$ as the $k \times \ell$ matrix valued $2$-form whose entry at the position $(i,j)$ is the $2$-form $da_{ij}$.
\vglue.1in

\begin{Example}
\label{Example:intersection} {\rm  Let ${\mathcal F}_1,\ldots,{\mathcal F}_q$ be transversely affine
codimension-one foliations on $M^n$, which are transverse everywhere.
Then the intersection foliation $\bigcap\limits_{i=1}^{q}\,
{\mathcal F}_j$ is a codimension-$q$ foliation on $M$ which is transversely affine. Indeed,
assume that ${\mathcal F}_j$ is given by some holomorphic integrable 1-form $\Omega_j$ in
$M$. According to \cite{scardua} Chapter I Proposition 1.1 we have $d\Omega_j = \eta_j \wedge \Omega_j$,
$d\eta_j=0$, for some holomorphic 1-form $\eta_j$ in $M$. Define $\Om$ as the $q\times 1$ matrix valued 1-form in $M$ having $\Omega_1,...,\Omega_q$ as rows.
Also define 
$\eta$ the $q\times q$  diagonal matrix valued  holomorphic 1-form in $M$ having $\eta_1,...,\eta_q$ in its diagonal. 
Then, in the above notation we have $d\Om = \eta \wedge\Om$. Since $\eta$ is diagonal, we have $ d \eta =0= \eta \wedge \eta$.}
\end{Example}

As for the general case we have the following description:

\begin{Theorem}
\label{Theorem:forms}
Let ${\mathcal F}$ be a holomorphic  codimension-$q$
foliation on $M$. The
foliation ${\mathcal F}$ is transversely affine in $M$ if, and only
if, there exist an open cover $\bigcup\limits_{i \in I} \,U_i=M$ and
holomorphic $q\times1$, $q\times q$ matrix valued {\rm 1}-forms
$\Omega_i$, $\eta_i$ in $U_i$, $\forall\, i \in I$, satisfying:

\noindent {\rm a)} ${\mathcal F}\big|_{U_i} = {\mathcal
F}(\Omega_i)$

\noindent {\rm b)} $d\Omega_i = \eta_i \wedge \Omega_i$ and $d\eta_i
= \eta_i \wedge \eta_i$

\noindent {\rm c)} if $U_i \cap U_j \ne\phi$ then we have $\Omega_i
= G_{ij}\cdot\Omega_j$ and $\eta_i = \eta_j + dG_{ij}\cdot
G_{ij}^{-1}$ for some holomorphic $G_{ij}\colon U_i\cap U_j \to
\GL_q(\co)$.

\noindent  Moreover, two such collections
$\{(\Omega_i,\eta_i,U_i)\}\_{i\in I}$ and
$\{(\Omega_i^\prime,\eta_i^\prime,U_i)\}\_{i\in I}$ define the same
affine transverse structure for ${\mathcal F}$, if and only if, we
have $\Omega_i^\prime = G_i\cdot\Omega_i$ and $\eta_i^\prime =
\eta_i + dG_i\cdot G_i^{-1}$ for some holomorphic $G_i\colon U_i \to
\GL_q(\co)$.
\end{Theorem}

\begin{Remark}
{\rm Theorem~\ref{Theorem:forms} is stated in a much more abstract context by Blumenthal (see Theorem 2 page  144 as well as its Corollary 3.2 page 149). Nevertheless, it is required some
triviality hypothesis on principal fiber-bundles of structural group
     $G/H$, over the manifold $M$ (see also \cite{Godbillon} Prop. 3.6 pp. 249-250).
In our case, we will obtain it from some explicit computations and some classical results on Lie groups (see Theorem~\ref{Theorem:Darboux-Lie}).}
\end{Remark}

In the final section we prove that an extension result for the pair $(\Omega, \eta)$ associate to an affine transverse structure off some codimension one divisor, under the presence of generic singularities for the foliation on the divisor (cf. Theorem~\ref{Theorem:extensionlemma}).

\section{Auxiliary results}

We state some results of easy proof which will be used in the proof
of  Theorem~\ref{Theorem:forms}.
We start by the following well-known lemma from real analysis, adapted to the holomorphic case:
\begin{Lemma}
\label{Lemma:basic}
 Let $X\colon U\subset \co^n \to \GL_q(\co)$ be a holomorphic map,
then $d(X^{-1}) = -X^{-1}\cdot dX\cdot X^{-1}$.
\end{Lemma}

Next step is:

\begin{Lemma}
\label{Lemma:eta} Let $X\colon U \subset \co^n \to \GL_q(\co)$ be  holomorphic and let $\eta$ be defined diagonal by $\eta
= dX\cdot X^{-1}$ then we have $d\eta= \eta \wedge \eta$. Given a holomorphic $q\times q$ matrix valued 1-form $\eta$ in $U\subset \mathbb C^n$, such that $d\eta= \eta \wedge \eta$, and a holomorphic map $G \colon U \to \GL_q(\mathbb C)$, then the 1-form $\tilde \eta := \eta + dG. G^{-1}$ satisfies $d \tilde \eta = \tilde \eta \wedge \tilde \eta$.
\end{Lemma}

\begin{proof}  Using Lemma~\ref{Lemma:basic} we have $d(X^{-1}) =-X^{-1}\cdot dX\cdot X^{-1}$. Thus
$$
\aligned
d\eta = d(dX\cdot X^{-1}) &= d(dX) \wedge X^{-1} +(-1) dX \wedge d(X^{-1})\\
&= (-1) dX \wedge (-X^{-1}\cdot dX\cdot X^{-1})\\
&=  (dX\cdot X^{-1}) \wedge  (dX\cdot X^{-1}) = \eta \wedge \eta.
\endaligned
$$
As for the second part, we have $ d \tilde \eta = d \eta + d(dG.G^{-1})=
\eta \wedge \eta + dG.G^{-1} \wedge dG.G{-1}$. On the other hand
$\tilde \eta \wedge \tilde \eta = (\eta + dG.G^{-1}) \wedge (\eta + dG.G^{-1}) = \eta \wedge \eta + \eta \wedge dG.G^{-1} + dG.G^{-1} \wedge \eta + dG.G^{-1} \wedge dG. G^{-1}= \eta \wedge \eta + dG.G^{-1} \wedge dG. G^{-1}$.
\end{proof}

Finally, we have:
\begin{Lemma}
\label{Lemma:glueing} Let $G,G^\prime\colon U \subset \co^n \to \GL_q(\co)$ be
holomorphic maps. Then we have $dG.G^{-1} = dG^\prime.G^{\prime -1}$
if and only if $G^\prime = G.A$ for some locally constant $A\colon U
\to \GL_q(\co)$.
\end{Lemma}

\begin{proof}  First we assume that $G^\prime = G\cdot A$ with $A$ locally
constant. Thus we have $G^{-1}\cdot G^\prime = A$ and therefore
$d(G^{-1}\cdot G^\prime) = dA = 0$ in $U$. This implies
$d(G^{-1})\cdot G^\prime + G^{-1}\cdot d(G^\prime) = 0$. Using that
$d(G^{-1})=-G^{-1}\cdot dG\cdot G^{-1}$  we have
$$
-G^{-1}\cdot dG\cdot G^{-1}\cdot G^\prime + G^{-1}\cdot dG^\prime =
0.
$$
Multiplying on the left this equality by $G$ we obtain
$$
-dG\cdot G^{-1}\cdot G^\prime + dG^\prime = 0.
$$
Multiplying on the right this last equality by $G^{\prime-1}$ we
obtain
$$
-dG\cdot G^{-1} + dG^\prime\cdot(G^\prime)^{-1} = 0,
$$
which proves the first part. Now we assume that $dG\cdot G^{-1} =
dG^\prime\cdot(G^\prime)^{-1}$ in $U$. Define $A = G^{-1}\cdot
G^\prime$ so that $G^\prime = G\cdot A$. We only have to show that
$dA=0$ in $U$.

\noindent  In fact, we have
$$
d(A) = d(G^{-1}\cdot G^\prime) = d(G^{-1})\cdot G^\prime +
G^{-1}\cdot d(G^\prime).
$$
Since  $d(G^{-1})=-G^{-1}\cdot dG\cdot G^{-1}$ we get
$$
\aligned
dA &= -G^{-1}\cdot dG\cdot G^{-1}\cdot G^\prime + G^{-1}\cdot dG^\prime\\
&= -G^{-1}\cdot (dG\cdot G^{-1} - dG^\prime\cdot
G^{\prime-1})G^\prime.
\endaligned
$$
Using the hypothesis $dG\cdot G^{-1} = DG^\prime\cdot G^{\prime-1}$
we obtain $dA=0$.
\end{proof}

Let $G$ be a Lie group and $\{\omega_1,...,\omega_\ell\}$ be a basis
of the Lie algebra of  $G$. Then we have $d\omega_k =
\sum\limits_{i<j} c_{ij} ^k \omega_i \wedge \omega_j$ for a family
constants $\{c_{ij}^k\}$ called the {\it structure constants} of the
Lie algebra in the given basis (\cite{Godbillon}). With this we have the classical theorem due to Darboux and Lie below. In few words, it says that maximal rank systems of 1-forms satisfying the same equations are locally pull-back of the group Lie algebra. The map is unique up to left translations in the Lie group.

\begin{Theorem}[Darboux-Lie, \cite{Godbillon}]
\label{Theorem:Darboux-Lie} Let  $G$ be a (complex) Lie group of dimension $\ell$. Let
 $\{\omega_1,...,\omega_\ell\}$ be a basis of the Lie algebra of
 $G$ with  structure constants $\{c_{ij}^k\}$.
Given a maximal rank system of (holomorphic) 1-forms
$\Omega_1,...,\Omega_\ell$ in a (complex) manifold $V$, such that $d\Omega_k=\sum_{i,j}^k
c_{ij}^k\, \, \Omega_i \wedge \Omega_j$, then:

\begin{enumerate}

\item For each point $p\in V$ there is a neighborhood
$p\in U_p \subseteq V$ equipped with a (holomorphic) submersion $f_p\colon U_p \to
G$ which defines $\fa$ in $U_p$ such that $f_p^*
(\omega_j)=\Omega_j$ in $U_p$, for all $j\in \{1,...,q\}$.

\item If $V$ is simply-connected we can take $U_p = V$.

\item  If $U_p \cap U_q \ne \emptyset$ then in the
intersection we have $f_q = L_{g_{pq}}(f_p)$ for some locally
constant left translation $L_{g_{pq}}$ in $G$.

\end{enumerate}

\end{Theorem}

\section{Transversely affine
foliations and differential forms}

The first step in the proof of Theorem~\ref{Theorem:forms} is:

\begin{Proposition}
\label{Proposition:formsglobal} Let ${\mathcal F}$ be a holomorphic codimension-$q$
foliation on $M$. Suppose that ${\mathcal F}$ is defined by some
integrable system $\{\Omega_1,\ldots,\Omega_q\}$ of holomorphic {\rm
1}-forms. If $\fa$ is transversely affine then there is  a $q\times q$  matrix valued holomorphic {\rm
1}-form $\eta = (\eta_{ij})$ satisfying:
$$
d\Om = \eta\wedge\Om,\quad d\eta = \eta \wedge \eta  \qquad\text{where}\qquad \Om =
\begin{pmatrix} \Omega_1\\ \vdots\\ \Omega_q\end{pmatrix}
$$
\end{Proposition}

\begin{proof}
 Let $\{\Omega_1,\ldots,\Omega_q\}$ be an integrable
holomorphic system which defines ${\mathcal F}$ in $M$ and suppose
$\{Y_i\colon U_i \to \co^q\}\_{i\in I}$ is a transversal affine
 structure for ${\mathcal F}$ in $M$ with
$$
Y_i = A_{ij}Y_j+B_{ij} \quad\text{in}\quad U_i \cap U_j \ne \emptyset
\, \, \, (1)
$$
as in Definition~\ref{Definition:transvaffine}.

\noindent  Since the submersions $Y_i$ define ${\mathcal F}$ we can
write
$$
\Om = G_i.dY_i  (2)
$$
in each $U_i$, for some holomorphic $G_i\colon U_i \to \GL_q(\co)$.
Here $\Om =
\begin{pmatrix} \Omega_1\\ \vdots\\ \Omega_q\end{pmatrix}$.

\noindent  In each $U_i \cap U_j \ne \emptyset$ we have:
$$
G_idY_i = G_j dY_j (3)
$$
and as it follows from (1)
$$
G_j = A_{ij}.G_i\,. (4)
$$
According to Lemma~\ref{Lemma:glueing} this last equality implies:
$$
dG_j.G_j^{-1} = dG_i.G_i^{-1} (5)
$$
in each $U_i \cap U_j \ne \emptyset$.

\noindent  This allows us to define $\eta$ in $M$ by
$$
\eta\big|_{U_i} = dG_i.G_i^{-1}. (6)
$$
According to Lemma~\ref{Lemma:eta} we have $d\eta = \eta \wedge \eta$. We also have in each $U_i$
$$
\aligned
d\Om = d(G_idY_i) &= dG_i \wedge dY_i\\
&= dG_i.G_i^{-1} \wedge dY_i\\
&= dG_i.G_i^{-1}\wedge G_i dY_i\\
&= \eta \wedge \Om.
\endaligned
$$
 The pair $(\Om, \eta)$
satisfies the conditions of the statement.

\end{proof}

Now we  study the converse of the proposition above.

\begin{Proposition}
\label{Proposition:formsgeneral}  Let ${\mathcal F}$ be a holomorphic condimension-$q$
foliation on $M$. The
foliation ${\mathcal F}$ is transversely affine in $M$ if, and only
if, there exist an open cover $\bigcup\limits_{i \in I} \,U_i=M$ and
holomorphic $q\times1$, $q\times q$ matrix valued {\rm 1}-forms
$\Omega_i$, $\eta_i$ in $U_i$, $\forall\, i \in I$, satisfying:

\noindent {\rm a)} ${\mathcal F}\big|_{U_i} = {\mathcal
F}(\Omega_i)$

\noindent {\rm b)} $d\Omega_i = \eta_i \wedge \Omega_i$ and $d\eta_i
= \eta_i \wedge \eta_i$

\noindent {\rm c)} if $U_i \cap U_j \ne\phi$ then we have $\Omega_i
= G_{ij}\cdot\Omega_j$ and $\eta_i = \eta_j + dG_{ij}\cdot
G_{ij}^{-1}$ for some holomorphic $G_{ij}\colon U_i\cap U_j \to
\GL_q(\co)$.

\noindent  Moreover, two such collections
$\{(\Omega_i,\eta_i,U_i)\}\_{i\in I}$ and
$\{(\Omega_i^\prime,\eta_i^\prime,U_i)\}\_{i\in I}$ define the same
affine transverse structure for ${\mathcal F}$, if and only if, we
have $\Omega_i^\prime = G_i\cdot\Omega_i$ and $\eta_i^\prime =
\eta_i + dG_i\cdot G_i^{-1}$ for some holomorphic $G_i\colon U_i \to
\GL_q(\co)$.
\end{Proposition}

\noindent  In order to prove in details the proposition above we
explicitly calculate the Lie algebra of $\Aff(\co^q)$. We consider $\GL_q(\mathbb C)$ as an open subset of the vector space $M(q\times q, \mathbb C)$ of complex $q\times q$ matrices. Using this we have:

\begin{Lemma}
\label{Lemma:affineliealgebra}  The Lie algebra $\aff(\co^q)$ of $\Aff(\co^q)$ has a basis
given by $\Om = X\cdot dY$, $\eta = dX\cdot X^{-1}$ where $X \in
\GL_q(\co)$ and $Y \in \co^q$ are global coordinates. Furthermore we
have $d\Om = \eta\wedge\Om$, $d\eta = \eta\wedge\eta$.
\end{Lemma}

\begin{proof} We denote by $M(q\times q, \mathbb C)$ the linear space of $q\times q$ complex matrices.  Since $\GL_q(\co) \subset  M(q\times q, \mathbb C)\cong \co^{q^2}$ as an
open set, we have a natural global coordinate $X$ in $\GL_q(\co)$.
Let us denote by $Y$ the natural global coordinate in $\co^q$. Fixed
any element $(X_o,Y_o) \in \Aff(\co^q)$ it defines a left
translation by
$$
\aligned &L_{(X_o,Y_o)}\colon \GL_q(\co)\times\co^q \longrightarrow
\GL_q(\co)\times\co^q\\
&L_{(X_o,Y_o)}(X,Y) = (X_oX, X_oY + Y_o).
\endaligned
$$
Therefore given any vector $(V,W) \in
T_{(X_o,Y_o)}(\GL_q(\co)\times\co^q)$ we have
$\DL_{(X_o,Y_o)}(X,Y)\cdot(V,W) = (X_oV,X_oW)$. Therefore a basis of
the left-invariant vector fields in $\Aff(\co^q)$ is given by:
$$
\X = (X,X) = X\cdot\frac{\po}{\po X} + X\cdot\frac{\po}{\po Y} \in
T(\Aff(\co^q)) = \GL_q(\co)\times\co^q.
$$
Thus a basis of $\aff(\co^q)$ is given by the dual basis
$\{\Om,\eta\}$ of $\{\X\}$. This shows that
$$
\begin{cases}
\Om &= X\cdot dY\\
\eta &= dX\cdot X^{-1}
\end{cases}
$$
is a basis for $\aff(\co^q)$.

\noindent  It is now a straightforward calculation to show that
$d\Om = \eta\wedge\Om$ and $d\eta = \eta\wedge\eta$.
\end{proof}

\noindent  Using these two lemmas and  Darboux-Lie Theorem (Theorem~\ref{Theorem:Darboux-Lie}) or alternatively,  the book of Spivak (\cite{spivak} Chapter 10, Theorem 17 page 397, \, Theorem 18 page 398 and Corollary 19 page 400) we obtain:

\begin{Corollary}
\label{Corollary:localformetaomega}
\begin{itemize}
\item[{\rm(a)}]

Let $\eta$ be a holomorphic $q\times q$ matrix valued {\rm
1}-form in $M$ satisfying $d\eta = \eta\wedge\eta$. Then locally in
$M$ we have $\eta = dX\cdot X^{-1}$ for some holomorphic $X\colon
U\subset M \to \GL_q(\co)$. If $M$ is simply-connected we can choose
$U=M$. Moreover given two such trivializations $(X,U)$ and
$(\widetilde X,\widetilde U)$ with $U \cap \widetilde U \ne \emptyset$
connected then we have $\widetilde X = X\cdot A$ for some $X \in
\GL_q(\co)$.

\item[{\rm b)}]

Let $\Om$, $\eta$ be holomorphic $q\times1$, $q\times
q$ matrix valued {\rm 1}-forms in $M$ satisfying $d\Om = \eta\wedge\Omega$
and $d\eta = \eta\wedge\eta$. Then given any point $p \in M$ and
given any simply- connected open neighborhood $p \in U_p \subset M$
we have $\Om = X\cdot dY$, $\eta = dX\cdot X^{-1}$ for some
holomorphic $\pi_p = (X,Y)\colon U_p \to \GL_q(\co)\times\co^q$.
Furthermore in each connected component of $U_p \cap \widetilde
U_{\tilde p} \ne \emptyset$ we have $\pi_q = L\circ\pi_{\tilde p}$ for
some left-translation $L\colon \GL_q(\co)\times \mathbb C^q  \to \GL_q(\co)\times \mathbb C^q$. In
particular if $M$ is simply-connected we can choose $U_p = M$.

\end{itemize}

\end{Corollary}

The proof of Proposition~\ref{Proposition:formsgeneral} is now an easy consequence of
 Corollary~\ref{Corollary:localformetaomega} above and of the arguments used in the  proof of Proposition~\ref{Proposition:formsglobal}.

\begin{proof}[Proof of Proposition~\ref{Proposition:formsgeneral}]
Proposition~\ref{Proposition:formsglobal} shows that if $\fa$ is transversely affine in $M$ then we can construct collections $(\Omega_j, \eta_j)$ in open subsets $U_j\subset M$ covering $M$ as stated.
\noindent  Conversely assume that $(\Om, \eta)$ is a pair, where $\Omega$ defines $\fa$ in $M$, like in
the statement. Since $\eta$ is holomorphic and satisfies $d\eta = \eta \wedge \eta$ in $M$, there
exists an open cover $\bigcup\,U_i$ of $M$ there are holomorphic
$G_i\colon U_i \to \GL_q(\co)$ such that $\eta\big|_{U_i} =
dG_i.G_i^{-1}$ (Corollary~\ref{Corollary:localformetaomega} (a)).

\noindent  Now, from condition $d\Om = \eta \wedge \Om$ we have
$$
\aligned
d(G_i^{-1}.\Om) &= -G_i^{-1}\,dG_i G_i^{-1} \wedge \Om + G_i^{-1}\,d\Om\\
&= -G_i^{-1}\,\eta \wedge \Om + G_i^{-1}\,\eta \wedge \Om = 0
\endaligned
$$
and therefore $G_i^{-1} = dY_i$ for some holomorphic $Y_i\colon U_i
\to \co^q$ which is a submersion.

\noindent  Therefore we have $\Om = G_i\,dY_i$ in $U_i$. Moreover
according to Lemma~\ref{Lemma:glueing} we have $G_i^{-1}\,G_j = A_{ij}$ for some
locally constant $A_{ij}\colon U_i\cap U_j \to \GL_q(\co)$, in each
$U_i \cap U_j \ne \emptyset$.

\noindent  Therefore $G_i\,dY_i = \Om = G_j\,dY_j =
G_i\,A_{ij}\,dY_j$ so that $dY_i = A_{ij}\,dY_j = d(A_{ij}\,Y_j)$ in
each $U_i\cap U_j \ne \emptyset$ and thus $Y_i = A_{ij}\,Y_j + B_{ij}$
for some locally constant $B_{ij}\colon U_i\cap U_j \to \co^q$. This
shows that ${\mathcal F}$ is tranversely affine in $M$.
\end{proof}

Theorem~\ref{Theorem:forms} is now a straightforward consequence of Propositions~\ref{Proposition:formsglobal} and ~\ref{Proposition:formsgeneral}.

%\begin{Remark}[The case of a global defining system] {\rm
%Let ${\mathcal F}$ be a holomorphic codimension-$q$
%foliation on $M$. Suppose that ${\mathcal F}$ is defined by some
%integrable system $\{\Omega_1,\ldots,\Omega_q\}$ of holomorphic {\rm
%1}-forms in $M$. Then, by the above results, any
%transverse affine  structure for ${\mathcal F}$ is given  by a $q\times q$  %matrix valued holomorphic {\rm
%1}-form $\eta = (\eta_{ij})$ satisfying:
%$$
%d\Om = \eta\wedge\Om,\quad d\eta = \eta \wedge \eta  %\qquad\text{where}\qquad \Om =
%\begin{pmatrix} \Omega_1\\ \vdots\\ \Omega_q\end{pmatrix}
%$$
%Moreover two such pairs $(\Om,\eta)$ and $(\Om^\prime, \eta^\prime)$
%define the same affine product transverse structure for ${\mathcal
%F}$ if and only if we have $\Om^\prime = G.\Om$ and $\eta^\prime =
%\eta +dG.G^{-1}$ for some holomorphic $G\colon M \to \GL_q(\co)$.}
%\end{Remark}

\section{A suspension example}

\noindent  The following example generalizes Example 1.5 of Chapter
I in \cite{scardua}.

\begin{Example} {\rm We will define a transversely affine codimension-$q$
holomorphic foliation on a compact manifold by the suspension
method:

\noindent  Let $M$ be a complex manifold and let $w$ be a $q\times1$
holomorphic matrix valued 1-form on $M$, closed and satisfying $f^*w =
Aw$ for some biholomorphism $f\colon M \to M$ and some hyperbolic
matrix $A \in \GL_q(\co)$.
  Define $\Om$ and $\eta$ in the product
$M\times\GL_q(\co)$ by $\Om(x,T) = T.w(x)$ and $\eta(x,T) =
dT.T^{-1}$.

\noindent  Then we have
$$
\aligned
d\Om(x,T) &= dT \wedge w(x) + T\,dw(x) =\\
&= dT \wedge w(x) = dT.T^{-1} \wedge Tw(x) =\\
&= \eta(x,T) \wedge \Om(x,T)
\endaligned
$$
and also,
$$
\aligned
d\eta(x,T) &= d(dT.T^{-1}) = dT.T^{-1} \wedge dT.T^{-1} =\\
&= \eta(x,T) \wedge \eta(x,T).
\endaligned
$$
Moreover the biholomorphism $F\colon M\times \GL_q(\co) \to
M\times\GL_q(\co)$ defined by $F(x,T) = (f(x),T.A^{-1})$ satisfies
$$
F^*\Om = TA^{-1}f^*w = TA^{-1}Aw = Tw = \Om
$$
and
$$
F^*\eta = d(TA^{-1})\cdot(TA^{-1})^{-1} = dT.T^{-1} = \eta.
$$
Thus, by Theorem~\ref{Theorem:forms} the pair $\Om$, $\eta$ induces a codimension-$q$
non-singular holomorhic foliation $\widetilde{\mathcal F}$ which is
transversely affine in $M\times\GL_q(\co)$. This foliation
induces a codimension-$q$ non-singular foliation ${\mathcal F}$ on
the quotient manifold $V = (M\times\GL_q(\co)/\mathbb Z$ by the
action $\mathbb Z\times(M\times\GL_q(\co)) \to M\times\GL_q(\co)$,
$\eta,(x,T) \mapsto (f^n(x),T.A^{-n})$. This last foliation
${\mathcal F}$ inherits and affine transverse structure from
$\widetilde{\mathcal F}$.}
\end{Example}

%%%%%%%%%%%%%%%%%%%%%%%%%%%%%%%%%%%%%%%%%%%%%%%%%%%%%%%%%%%%%%%%%%%%%%%%
%%%%%%%%%%%%%%%%%%%%%%%%%%%%%%%%%%%%%%%%%%%%%%%%%%%%%%%%%%%%%%%%%%%%%%

\section{Holomorphic foliations with singularities}

A a codimension $q$ holomorphic foliation with singularities $\fa$  on a complex manifold $M$ of dimension $n\geq 2$ is defined as a pair $(\fa_0, \sing(\fa))$, where
$\sing(\fa)\subset M$ is an analytic subset of codimension $\geq q+1$, and a holomorphic foliation $\fa_0$ in the classical, in the open manifold $M\setminus\sing(\fa)$.
Then, all the notions for $\fa$ are defined in terms of $\fa_0$. For instance, the leaves of $\fa$ are defined as the leaves of $\fa_0$, and their holonomy groups are defined in the same way. We may assume that the {\it singular set} $\sing(\fa)$ is saturated in the sense that there is no other pair $\fa^\prime=(\fa_0 ^\prime, \sing(\fa^\prime)$ with $\sing(\fa^\prime)\subsetneqq \sing(\fa)$ and such that $\fa_0 ^\prime$ coincides with $\fa_0$ on $M\setminus \sing(\fa)$.

\begin{Definition}
\label{Definition:transvaffinesing} {\rm A  codimension-$q$ holomorphic foliation with singularities ${\mathcal F}$ on
$M^n$ is said to be {\it transversely affine\/} if there is a family
$\{Y_i\colon U_i \to \co^q\}\_{i\in I}$ of holomorphic submersions
$Y_i\colon U_i \to \co^q$ defined in open sets $U_i \subset M$,
defining ${\mathcal F}$, and satisfying $M\backslash \sing(\fa)
= \bigcup\limits_{i \in I} \,U_i$ and with affine relations $Y_i =
A_{ij}Y_j + B_{ij}$ for some $A_{ij}\colon U_i\cap U_j \to
\GL_q(\co)$, $B_{ij}\colon U_i\cap U_j \to \co^q$ locally constant
in each $U_i \cap U_j \ne\phi$. }
\end{Definition}

We usually distinguish two cases in the definition above: the codimension one and the dimension one cases. Since we are interested in the codimension $\geq 2$ case, we shall focus on the second case.

\subsection{Generic singularities}
In this paragraph we introduce what we will consider as generic type of a
singularity for a codimension-$q \ge 2$ foliation.
Given a  holomorphic foliation with singularities $\fa$ on a complex manifold $M$, the singular set  of $\fa$ is an analytic  subset $\sing(\fa) \subset M$ of codimension $\geq 2$, also having dimension $\dim \sing(\fa) \leq \dim(\fa)$. In particular, it can have a component of dimension $\dim(\fa)$, as well as a component of dimension $\dim(\fa) -1$. As for this second case, by intersecting with appropriate transverse small discs we may consider the following model of generic singularity:

\subsubsection{Isolated singularities}

\begin{Definition} {\rm Let $\fa$ be a germ of an {\em isolated}  one-dimensional foliation singularity  at the origin
$0\in\mathbb C^{q+1}$. The singularity is called {\it Poincar\'e non-resonant} if the convex hull of the set of eigenvalues of the linear part $DX(0)$ does not contain the origin, and there is no resonant $\lambda_j = n_1 \lambda _1 + ... n_{q+1} \lambda_{q+1}$ for $n_1,...,n_{q+1} \in \mathbb N$.
In this case, by Poincar\'e linearization theorem (\cite{[Brjuno]}, \cite{[Dulac]})
the singularity {\it linearizable
without resonances} (\cite{mafra-scardua}):  it is given in some neighborhood $U$ of $0\in\mathbb C^{q+1}$ by a holomorphic vector field $X$ which is
analytically linearizable as
$X={\displaystyle \sum_{j=1}^{q+1}\lambda_{j}z_{j}\dfrac{\partial}{\partial
z_{j}}},$ with eigenvalues $\lambda_{1},\cdots,\lambda_{q+1}$ satisfying the following non-resonance
hypothesis:

\noindent {\sl If $n_{1},\cdots,n_{q+1}\in\mathbb Z$ are such that
$\sum_{j=1}^{q+1}n_{j}\lambda_{j}=0,$
 then $n_{1}=n_{2}=\cdots=n_{q+1}=0.$
   }
}
\end{Definition}

In the above situation,
define  1-forms
$\omega^{1},\cdots,\omega^{q}$ on $U\setminus\Lambda$ by setting
$\omega^{\nu}(X)=0$ and
$\omega^{\nu}=\sum_{j=1}^{q+1}\alpha_{j}^{\nu}\frac{dz_{j}}{z_{j}}$, 
 where $\nu=1,\cdots,q$ and $\alpha_{j}^{\nu}\in\mathbb C$. From this we
get the following system of equation
$\sum_{j=1}^{q+1}\alpha_{j}^{\nu}\lambda_{j}=0,\quad\nu=1,\cdots,q.$
the equation
$\sum_{j=1}^{q+1}\lambda_{j}z_{j}=0$
 defines a hyperplane in $\mathbb C^{q+1}$ implies that we can choose $q$
linearly independent vectors $\vec{\alpha}_{1},\cdots,\vec{\alpha}_{q}$
say
$\vec{\alpha}_{\nu}=(\alpha_{1}^{\nu},\cdots,\alpha_{q}^{\nu},\alpha_{q+1}^{\nu})\in\mathbb C^{q+1}$
so that
$\sum_{j=1}^{q+1}\alpha_{j}^{\nu}\lambda_{j}=0,\quad\nu=1,\cdots,q.$
 and therefore the system $\omega^{1},\cdots,\omega^{q}$ has maximal rank
$q$ outside the coordinate hyperplanes.

\begin{Lemma}[\cite{mafra-scardua}]
\label{Lemma:nonresonantconstant}
Let $f(z)$ be a holomorphic function on the set
$U\setminus\left\{ z_{1}\cdot\ldots\cdot z_{q+1}=0\right\} $, where
$U$ is a connected neighborhood of the origin in $\mathbb C^{q+1}$. Then $f(z)$ is
constant provided that
$df\wedge\omega^{1}\wedge\cdots\wedge\omega^{q}=0$.
\end{Lemma}

%%%%%%%%%%%%%%%%%%%%%%%%%%%%%%%%%%%%%%%%%%%%%%%%%%%%%%%%%%%%%%%%%%%%

\begin{Definition}[type II generic singularities]
\label{Definition:isolatedtype}
{\rm
A singularity $p \in \sing(\fa)$ will be called {\it type II generic singularity} if $p$ belongs to a smooth part of the set $\sing(\fa)$, where:
 \begin{itemize}
 \item There is a unique  branch $\sing(\fa)_p\subset \sing(\fa)$ through $p$.
  \item $\dim \sing(\fa)_p = \dim (\fa) - 1$
  \item For some (and therefore for every) transverse disc $\Sigma_p$, with $\Sigma_p \cap \sing(\fa)_p = \Sigma_p \cap \sing(\fa) = \{p\}$, of dimension $q +  1$, the induced foliation $\fa\big|_{\Sigma_p}$ exhibits an isolated non-resonant Poincar\'e type singularity at the origin $p$.
      \end{itemize}
      }
\end{Definition}

 \subsubsection{Non-isolated singularities} Now we focus on the components of
the singular set that cannot be reduced to isolated singularities by transverse sections.  Let us first recall that some notions for codimension one foliations. Given a codimension-one holomorphic foliation with singularities $\fa$ on a complex manifold $M$,  a singular point $p \in \sing(\fa)$ is a {\it Kupka-type singularity} (cf. \cite{omegar, scardua}), if $\fa$ is given in some neighborhood $U$ of $p$ by a holomorphic integrable 1-form $\omega$, such that $\omega(p)=0, \, d \omega( p ) \ne 0$.
In this case, if $U$ is small enough, there exists a system of local coordinates
$(x,y,z_1,\ldots,z_{n-2}) \in U$ of $M$, centered at $p$, such that
${\mathcal F}\big|_U$ is given by
$\alpha(x,y)=0$, for some holomorphic 1-form $\alpha= A(x,y) dx + B(x,y) dy$.  The 1-form $\alpha$, so called the {\it transverse type of $\fa$ at $p$},  has an isolated singularity at the origin $0 \in \mathbb C^2$ and satisfies $d \alpha (0)\ne 0$.
The generic type is then defined as follows:
 We shall say that a singularity $p \in\sing(\fa)$ is {\it Poincar\'e type} if it is Kupka type and its corresponding transverse type is of the form $xdy-\lambda ydx + hot=0,\qquad \la \in \mathbb C\backslash(\mathbb R_-\cup \mathbb Q_+)$.
The reasons for this are based on the classification of singularities of germs of foliations in dimension two (see \cite{seidenberg}, \cite{camacho-linsneto-sad}).
In this case, the singularity $\alpha(x,y)=0$ is analytically linearizable, so that
there are coordinates $(x,y,z_1,\ldots,z_{n-2})$ as above, such that $\fa$ is given in these coordinates by $xdt - \lambda y dx=0$. Let us now motivate our second type of generic singularity for codimension $q \geq 2$ foliations, by
discussing an example:
\begin{Example}{\rm
Let ${\mathcal F}_1,\ldots,{\mathcal F}_q$ be holomorphic singular codimension
one foliations on a complex manifold $M$ of dimension $q+1$. Assume that the foliations $\fa_j$  are transverse outside the union of
their singular sets and their set of tangent points. Then we can define
in the natural way the {\it intersection foliation\/} ${\mathcal F}
= \bigcap \limits_{j=1}^{q}\,{\mathcal F}_j$ (as in Example~\ref{Example:intersection}) whose leaves are
obtained as the connected components of the intersection of the
leaves of ${\mathcal F}_1,\ldots,{\mathcal F}_q$ through points of
$M$ and has singular set $\sing(\fa)
=\bigcup\limits_{j=1}^{q}\,\sing({\mathcal F}_j) \cup T_2$ where $T_2$
is the union of the codimension $\ge 2$ components of the set of
tangent points of the foliations.
Suppose that ${\mathcal F}_j$ has only Poincar\'e type singularities, as defined above. Then, given any point $p
\in \sing({\mathcal F}_j)\backslash\bigcup\limits_{ i \ne j} \,\sing({\mathcal F}_i)$, there exists a
local chart $(x,y,z_1,\ldots,z_{n-2}) \in U$ of $M$, centered at $p$, such that
${\mathcal F}_j\big|_U$ is given by
$$
xdy-\la ydx=0,\qquad \la \in \co\backslash(\re_-\cup \mathbb Q_+)
$$
and for each $i\ne j$, ${\mathcal F}_i\big|_U$ is regular given by $dz_{k_i}=0$ for some
$k_i \in \{1,\ldots,n-2\}$.

}

\end{Example}

\begin{Definition}[type I generic singularities]
\label{Definition:intersectiontype}
{\rm  Let ${\mathcal F}$ be a codimension-$q$ foliation on $M^n$. A singularity $p \in \sing({\mathcal F})$ is a  {\it type I generic singularity}, if $p$ belongs to a smooth part of the set $\sing(\fa)$, where:
 \begin{itemize}
 \item There is a unique  branch $\sing(\fa)_p\subset \sing(\fa)$ through $p$.
  \item $\dim \sing(\fa)_p  =  \dim (\fa)$
\item  There is a local chart $(x,y,z_1,\ldots,z_{n-2}) \in U$ of
$M$, centered at $p$, such that ${\mathcal F}\big|_U$ is given by
$$
xdy -\la y dx = 0, \qquad \la \in \co\backslash(\mathbb R_- \cup \mathbb Q_+)
$$
and $dz_j=0$, $j=1,\ldots,q-1$.
\end{itemize}
\noindent  Therefore in a neighborhood of $p$, the foliation ${\mathcal F}$ has the
structure of the {\em intersection} (not product) of a singular linear  foliation $xdy - \la y dx = 0$ on $(\co^2,0)$ and $q-1$ regular  trivial foliations.

\noindent  We have $s({\mathcal F}) \cap U = \{(x,y,z_1,\ldots,z_{n-2}) \in U \mid
x=y=0\}$. If we define $\La = \{xy=0\} \cap \{z_1 =\cdots= z_{q-1} = 0\}$ then
$\La$ consists of two codimension-$q$ invariant local submanifolds $\La_1 \cup
\La_2$ which intersect transversely at the point $p = \La_1 \cap \La_2$.}
\end{Definition}

\section{Extending affine transverse structures with poles}

\noindent  Now consider the following situation:
\begin{enumerate}

\item ${\mathcal F}$ is a codimension-$q$ singular foliation on $M$,
 \item $\La \subset M$ is an analytic irreducible invariant subvariety of
codimension-$q$ (i.e., $\La\backslash \sing(\fa)$ is a leaf of
${\mathcal F}$),
\item   There are analytic codimension-one subvarieties
$S_1,\ldots,S_q \subset M$ such that $\La$ is an irreducible
component of $\bigcap\limits_{j=1} ^q \,S_j$ and $S_j$ is foliated
by ${\mathcal F}$, $j = 1,\ldots,q$.
\end{enumerate}

\noindent  Under these assumptions we make the following definition:

\begin{Definition}
\label{Definition:adaptedeta}{\rm  Let $\{\Omega_1,\ldots,\Omega_q\}$ be an integrable system of
holomorphic 1-forms defining $\fa$. A $q\times q$ matrix valued {\sl meromorphic} 1-form  $\eta$ defined in a neighborhood of $\La$ is said
to be {\it a partially-closed logarithmic derivative adapted to $\Om$ along $\La$\/} if:
\begin{itemize}

\item $d\Omega= \eta\wedge \Omega$ and $\eta$ is partially-closed, $d \eta = \eta \wedge \eta$, meromorphic with simple poles,
\item  $(\eta)_\infty = \bigcup\limits_{j=1} ^q \,S_j$, a union of irreducible codimension one analytic subsets $S_j\subset V$ in a neighborhood $V$ of
$\La$,
 \item given any regular point $p \in \La\backslash \sing(\fa)$ there exists a local chart $(y_1,\ldots,y_q,
z_1,\ldots,z_{n-q}) \in U$ for $M$, centered at $p$, such that:
$$
\aligned
&U \cap S_j = \{y_j=0\}, \quad j = 1,\ldots,q\\
&\Om = G.dY \qquad\text{and}\\
&\eta = dG.G^{-1}+\sum\limits_{j=1} ^q A_j.\frac{dy_j}{y_j}
 \qquad\text{where}\\
&Y = \begin{pmatrix} y_1\\ \vdots\\ y_q\end{pmatrix}\,,
\endaligned
$$
$G\colon U \to \GL_q(\co)$ is holomorphic  and $A_j$ is a
constant $q\times q$ complex  matrix.
\end{itemize}

\noindent  The  matrix $A_j$ is called the {\it residue
matrix\/} of $\eta$ with respect to $S_j$.
}
\end{Definition}

In what follows we consider the problem of extending a form $\eta$ from an affine transverse structure of $\fa$,   an analytic invariant
hypersurface.
The existence of such extension, as adapted closed logarithmic derivatives, is then assured by the following result:

\begin{Theorem} [Extension Lemma]
\label{Theorem:extensionlemma} Let ${\mathcal F}, \, \Lambda$ be as above. Suppose:

\item{{\rm (1)}} $\sing(\fa) \cap \La$ is nonempty and  consists of type I and type II generic singularities, and  singularities where $\dim\sing(\fa) \leq \dim(\fa) - 2$.

\item{{\rm (2)}} There exists a differential {\rm 1}-form $\eta$ defined in
some neighborhood $V$ of $\La$ minus $\La$ and its local separatrices
which defines a transverse affine structure for ${\mathcal F}$ in
this set $V\setminus (\Lambda \cup \sep(\Lambda))$, in the sense of Proposition~\ref{Proposition:formsglobal}.

\noindent  Then $\eta$ extends meromorphically to a neighborhood of
$\La$ as an adapted form (in the sense of Definition~\ref{Definition:adaptedeta}) to $\Om$
along $\La$.
\end{Theorem}

 We will extend $\eta$ to $\La$ through the singularities of $\fa$ in $\Lambda$. According to classical Hartogs' extension theorem (\cite{GunningII,Gunning-Rossi}), this implies the extension to $\Lambda$. Choose
$p \in \sing(\fa) \cap \La$ and choose local coordinates
$(x,y,z_1,\ldots,z_{n-2}) \in U$, centered at $p$, as in Definition~\ref{Definition:intersectiontype}.

\begin{Lemma}
\label{Lemma:eta_0}  Let $\mathcal F$ be a codimension $q$  holomorphic
foliation with singularities, defined in an open polydisc $U\subset \mathbb C  ^{q+n}$, with a type I generic singularity or a type II generic  singularity at  the origin $0 \in \sing
({\mathcal F})\subset U$. Assume that  ${\mathcal F}$ is transversely affine in
$U\setminus \Lambda$, where $\Lambda\subset U$ is a finite union of
irreducible invariant  hypersurfaces, each one containing the origin.
Assume that $\mathcal F$ is given in $U$ by a  holomorphic $q\times
1$ matrix $1$-form $\Omega$ in $U$ with a $q\times q$ matrix {\rm
1}-form $\eta$ in $U$ satisfying:
\[
d\Omega = \eta \wedge \Omega, \, \, \,  d\eta = \eta \wedge \eta.
\]

\noindent  Then $\eta$ extends meromorphically to a neighborhood of
$\La$ as a partially-closed logarithmic derivative adapted to $\Omega$ along $\Lambda$  (in the sense of Definition~\ref{Definition:adaptedeta}).

\end{Lemma}

\begin{proof}
For the sake of simplicity of the notation we will assume that
$\mathcal F$ has codimension $q=2$ and the ambient has dimension
$q+1=3$. Let us also assume that the singularity is isolated, i.e., of non-resonant Poincar\'e of  type II. The general case is pretty similar. Let then $X=\sum_{j=1}^3
\lambda _j \,  x_j \, ({\partial}/{\partial x_j})$ be a holomorphic
vector field defining $\mathcal F$ in suitable coordinates $(x_1,
x_2, x_3) \in U^\prime$, in a connected neighborhood $0 \in
U^\prime\subset U$ of $0\in {\mathbb C}^3$, with
$\{\lambda_1,\lambda_2,\lambda_3\}$ linearly independent over
$\mathbb Q$.
 Given complex numbers $a_1,a_2,a_3$ we  define a closed 1-form $\omega=\sum\limits
 _{k=1}^3 a_k \, dx_k/x_k$. Then $\omega(X)=0$ if and only if $\sum\limits_{k=1}^3
 a_k \, \lambda_k=0$.
Thus, we can choose 1-forms  $\omega_1,\omega_2$  given by
$\omega_j=\sum_{j=1}^3 a_k^j \, \, {dx_k}/{x_k}, \, a_k ^j \in
{\mathbb C}$, such that: $\omega_1$ and $\omega_2$  are linearly
independent in the complement of $\cup_{j=1}^3 (x_j=0)$ and
$\Theta_j(X)=0, j=1,2$.

Once we fix such 1-forms, the foliation $\mathcal F$ is defined by
the integrable system of meromorphic 1-forms $\{\omega_1,
\omega_2\}$ in $U$. Notice that the polar set of the $\omega_j$ in
$U^\prime$ consists of the coordinate hyperplanes $\{x_i=0\}\subset
U^\prime, \, i=1,2,3.$ \, Let $\Omega_0$ be the $2\times 1$ meromorphic matrix valued  1-form given by the system $\{\omega_1, \omega_2\}$.
\begin{Claim}
\label{Claim:eta_0}
Let $\eta_0$ be a  $2\times 2$ holomorphic
matrix valued 1-form defined in $U^\prime \setminus \bigcup
\limits_{i=1}^3 \{x_i=0\}$, such  that
$d\Omega_0 = \eta_0 \wedge \Omega_0, \, \, \,  d\eta_0 = \eta
_0\wedge \eta_0.$
Then:
\begin{enumerate}
 \item $\eta_0$ is closed, $d \eta_0 = 0$.
 \item The matrix  valued 1-form $\eta_0$ extends to a meromorphic matrix valued 1-form
in $U^\prime$,   having polar divisor of order one in $U^\prime$.
\item The extension of $\eta_0$ is adapted to $\Omega_0$ along $\Lambda$.
\end{enumerate}
\end{Claim}

Let us see how the claim proves the lemma.
Indeed, as for the original forms $\Omega$ and $\eta$ we have $\Omega =
G\Omega_0$  for some holomorphic matrix  $G\colon \widetilde U \to
\GL_q(\co)$. Thus if we define $\eta_0:= \eta -  dG\cdot G^{-1}$
then we are in the situation of the above claim. Thus we conclude
that $\eta$ extends to $U ^\prime$ as a closed meromorphic 1-form
with simple poles and polar divisor consisting of the coordinate
planes. Therefore, the same conclusion of the above claim holds for
$\eta$ and we prove the lemma.

\begin{proof}[Proof of the claim]
Since each $\omega_j$ is closed the matrix form $\Omega_0$ is
closed. From $d\Omega_0 = \eta_0 \wedge \Omega_0$ we have $\eta_0
\wedge \Omega_0=0$.
Now we observe  that there are holomorphic $2 \times 2$ scalar
matrices  $M_1, M_2$  defined in $U^\prime \setminus \{x_1 x_2 x_3
=0\}$, such that $\eta_0= M_1 \omega_1 + M_2 \omega_2$, where the
multiplication of the matrix by the 1-form is the standard scalar
type multiplication. Indeed, it is enough to complete the pair $\omega_1, \omega_2$ into  a basis of the space of holomorphic 1-forms and express $\eta_0$ in this basis. Then the condition $\eta_0 \wedge \Omega_0$ means that the coefficients of $\eta_0$ in the other elements of the basis are all identically zero.

For  any holomorphic  $2\times 2$ scalar (holomorphic) matrix $M$ and a $2
\times 1$ matrix valued 1-form $\Omega$ we have the easily verified formula for
the exterior derivative:
\[
d (M\Omega) = dM \wedge \Omega + M d \Omega
\]

Therefore  we have
\[
d\eta_0= dM_1 \wedge \omega_1 + d M_2 \wedge \omega_2.
\]

Also of easy verification we have
\[
\eta_0 \wedge \eta_0 = [M_1, M_2] \omega_1 \wedge \omega_2
\]
where $[,]$ denotes the matrix Lie bracket. Thus we obtain
\[
 dM_1 \wedge \omega_1 + d M_2 \wedge \omega_2= [M_1, M_2] \,\omega_1 \wedge \omega_2.
 \]

Taking the exterior product with $\omega_2$ in the  above equation
we obtain
\[
 dM_1 \wedge \omega_1 \wedge \omega_2=0
 \]
 Hence, $M_1$ is a meromorphic first integral for the foliation defined by
the system $\{\omega_1, \omega_2\}$ in $\widetilde U:= U^\prime
\setminus \{x_1 x_2 x_3=0\}$. This foliation is exactly the
restriction of $\mathcal F$ to this open set. Since $\mathcal F$ is
defined by the vector field $X$ in $\widetilde U$ and this vector
field is linear without resonance, it follows from Lemma~\ref{Lemma:nonresonantconstant}  that $M_1$ is constant in $\widetilde U$. Similarly we can
conclude that $M_2$ is constant. This implies the extension result
and the other items in Claim~\ref{Claim:eta_0}.
\end{proof}
The proof of Lemma~\ref{Lemma:eta_0} is complete.
\end{proof}

\begin{proof}[Proof of Theorem~\ref{Theorem:extensionlemma}]
The proof follows the same argumentation 
as the proof of Lemma 3.2 in Chapter I in \cite{scardua}. 
Indeed,  Lemma~\ref{Lemma:eta_0}  implies that  $\eta$ extends
meromorphically to $\La \cup \text{sep}\,(\La)$. By construction  this extension
is adapted to $\Om$ along $\La$. \end{proof}

%%%%%%%%%%%%%%%%%%%%%%%%%%%%%%%%%%%%%%%%%%%%%%%%%%%%%%%%%%%%%%%%%%%%%%%%%%%%%%%%%%
%%%%%%%%%%%%%%%%%%%%%%%%%%%%%%%%%%%%%%%%%%%%%%%%%%%%%%%%%%%%%%%%%%%%%%%%%%%%%%%%%%%%%%%

\bibliographystyle{amsalpha}

\end{document}